\documentclass[oneside,a4paper,12pt]{amsart} 
\usepackage{amssymb}
\usepackage{amscd}
\usepackage{tikz-cd}
\usepackage[all,cmtip]{xy}
\usepackage{verbatim}
\usepackage{lscape}
\usepackage{enumerate}
\theoremstyle{plain}
\newtheorem{theorem}{Theorem}[section]
\newtheorem{lemma}[theorem]{Lemma}

\newtheorem{proposition}[theorem]{Proposition}

\theoremstyle{definition}

\newtheorem{definition}[theorem]{Definition}
\newtheorem*{exercises}{Exercises}

\theoremstyle{remark}
\newtheorem*{remark}{Remark}
\newtheorem*{notation}{Notation}
\newtheorem*{example}{Example}

\DeclareMathOperator{\Aut}{Aut} \DeclareMathOperator{\Ext}{Ext}
\DeclareMathOperator{\Hom}{Hom} \DeclareMathOperator{\res}{res}
\DeclareMathOperator{\Ker}{Ker} 
 \DeclareMathOperator{\Stab}{Stab}
\DeclareMathOperator{\Res}{Res} 

\DeclareMathOperator{\End}{End} 
 
\DeclareMathOperator{\Proj}{Proj} 
\DeclareMathOperator{\inj}{inj}
\DeclareMathOperator{\free}{free}
\DeclareMathOperator{\Mod}{Mod}

\DeclareMathOperator{\dirlim}{\varinjlim}
\DeclareMathOperator{\indec}{indecomposable}
\DeclareMathOperator{\proj}{proj} 

\DeclareMathOperator{\ev}{ev}
\DeclareMathOperator{\fd}{finite\ dimensional}
\DeclareMathOperator{\rank}{rank}
\DeclareMathOperator{\np}{np} 
\DeclareMathOperator{\anything}{anything}
\DeclareMathOperator{\projdim}{projdim}

\DeclareMathOperator{\findim}{findim}
\DeclareMathOperator{\cd}{cd}
\DeclareMathOperator{\vcd}{vcd}

\DeclareMathOperator{\Sl}{Sl}
\DeclareMathOperator{\ModProj}{ModProj}
\DeclareMathOperator{\Grp}{Grp}

\mathchardef\mhyphen="2D

\begin{document}

\title{Endotrivial Modules for Infinite Groups}

\author{Peter Symonds}
\address{School of Mathematics\\
         University of Manchester\\
     Manchester M13 9PL\\
     United Kingdom}
\email{Peter.Symonds@manchester.ac.uk}

%\subjclass{Primary: 13A50; Secondary: 13D45, 20C05}
%\keywords{keywords}
%\subjclass{Primary: subject; Secondary: subject}
%\date{}

%\begin{abstract}
%\end{abstract}

\maketitle

%%%%%%%%%%%%%%%%%%%%%%%%%%%%%%%%%%%%%%%%%%%%%%%%%%%%%%%%%%%%%%%%%%%%%%%%%%%%%%%%

These notes were originally intended for a course given at the PIMS Summer School on Geometric and Topological Aspects of the Representation Theory of Finite Groups in Vancouver, July 27--30 2016. The material is based on a project with Nadia Mazza and a more general and complete treatment of the new results will appear elsewhere. We intend these notes to serve as an introduction to the field.

Why study endotrivial modules for infinite groups?

\begin{enumerate}

\item
You cannot say much about all modules. Look for some small class where you might be able to say something interesting.

\item 
For finite groups, endotrivial modules or their generalization, endopermutation modules, occur as sources of simple modules for $p$-solvable groups and in the description of the source algebra of a nilpotent block. Their classification for finite $p$-groups was a major achievement.

\item
We are led to reconsider a lot of work from the 1970s and '80s on cohomology of groups.

\item
It forces us to look carefully at stable categories of modules for infinite groups and suggests ways these might be described.

\end{enumerate}

What we are going to do is as follows.

\begin{enumerate}

\item
Recall briefly the definition and properties of endotrivial modules for finite groups in the way that we will use them. These were covered in other courses in the Summer School.

\item
Define a class of infinite groups that we call groups of type $\Phi$. These are the groups that our methods can deal with and we investigate their homological properties.

\item
Construct the category of stable $kG$-modules for $G$ of type $\Phi$. This is where our endotrivial modules will live.

\item
Define endotrivial modules for groups of type $\Phi$ and develop their basic properties.

\item
Prove the existence of an exact sequence that makes it possible to calculate the group of endotrivial modules for an amalgamated free product or an HNN extension.  

\end{enumerate}

For background material, it is helpful to know a little about endotrivial modules for finite groups. We recommend the article by Th\'{e}venaz \cite{thevenaz}. We assume that the reader is familiar with basic homological algebra and has some knowledge of derived and triangulated categories. One possible reference is the book by Happel \cite{happel}.

\section{Finite Groups}
\label{sec:fg}
 
\begin{notation}
In these notes $k$ will always be a field of finite characteristic $p$. In fact, everything we do can be carried through for $k$ any commutative noetherian ring of finite global dimension (and $p$-local when we mention $p$). 
\end{notation}

For this section, $G$ will be a finite group and all $kG$-modules will be finite dimensional unless stated otherwise. If $M$ and $N$ are two $kG$-modules then $M \otimes _kN$ is naturally a $kG$-module under the action $g(m \otimes n)=gm \otimes gn$. All tensor products will be over $k$ unless otherwise indicated.

\begin{definition}
A $kG$-module $M$ is endotrivial if there is another module $N$ such that $M \otimes  N \cong k \oplus (\proj)$. 
\end{definition}

Here $(\proj)$ denotes some projective $kG$-module. We write $M \simeq M'$ when $M \oplus (\proj) \cong M' \oplus (\proj)$ and denote the equivalence class by $[M]$. Note that $(\proj) \otimes_k (\anything) \cong (\proj)$, so the equivalence classes form an abelian group under $\otimes_k$, which we denote by $T(G)$. 

Notice that the trivial module $k$ is the identity element and $N$ in the definition is the inverse of $M$. Each equivalence class contains an indecomposable module $M$ such that every other member of the class is of the form $M \oplus (\proj)$ (exercise).

Given a $kG$-module $M$, find a surjection from a projective module to $M$, $P \twoheadrightarrow M$ and let $\Omega M$ be the kernel, so we have a short exact sequence $\Omega M \hookrightarrow P \twoheadrightarrow M$.  Then $\Omega M$ is well defined up to projective summands (Schanuel's Lemma), so $[\Omega M]$ is well defined. We can also go in the other direction using injective modules: $M \hookrightarrow I \twoheadrightarrow \mho M$. For finite groups, injective is equivalent to projective, so we write $\Omega^{-1}$ instead of $\mho$. Iterating these gives us $[\Omega ^r M], \ r \in \mathbb Z$.
It is an easy exercise to check that $\Omega (M \otimes  N) \simeq (\Omega M) \otimes  N$.

Now suppose that $M$ is endotrivial, so $M \otimes  N \simeq k$. Then $\Omega M \otimes  \Omega^{-1} N \simeq \Omega(M \otimes \Omega^{-1} N) \simeq M \otimes \Omega \Omega^{-1} N \simeq M \otimes N \simeq k$, so $\Omega M$ is also endotrivial.
Clearly $k$ is endotrivial, so $\Omega^r M$ is endotrivial; since $\Omega^r k \otimes \Omega^s k \simeq \Omega^{r+s} k$, we obtain a homomorphism $\mathbb Z \rightarrow T(G), \quad r \mapsto \Omega^r k$.
$T(G)$ also contains all 1-dimensional representations of $G$.

There are natural restriction maps $T(G) \rightarrow T(H)$ for $H \leq G$. 

Recall that $\Hom_k(M,N)$ is considered to be a $kG$-module by setting $(gf)(m)=gf(g^{-1}m)$ for $f \in \Hom_k(M,N), g\in G, m\in M$. In particular, we have the dual $kG$-module, $M^*=\Hom_k(M,k)$.

\begin{lemma} 
\label{lem:invdual}
If $M \otimes  N \cong k \oplus (\proj)$ then $[N] = [M^*]$. The natural evaluation map $\ev: M \otimes  M^* \rightarrow k$ is split over $kG$ and the kernel is projective.

\end{lemma}

\begin{proof}

We have a commutative diagram

\[ \begin{tikzcd}
M \otimes N \ar[r, "\cong"] \ar[d, "1 \otimes \phi"]
& k \oplus (\proj) \ar[r, "\pi"] &k \ar[d, equal] \ar[l, dashed, bend right]\\
M \otimes M^* \ar[rr, "\ev"] & & k,
\end{tikzcd} \]
where $\phi : N \rightarrow M^*$ is given by $(\phi (n))(m)=\pi(m \otimes n)$. Thus $k \mid M \otimes M^*$, splitting $\ev$ (the notation means \emph{is a summand of}) and so $N \mid N \otimes M \otimes M^* \cong M^* \oplus(\proj)$. But $M = M' \oplus (\proj)$ for some indecomposable module $M'$ and so $N \cong M'^* \oplus (\proj)$. Note that $N$ is not projective if $p \mid |G|$ and if $p  \nmid |G|$ then all modules are projective and there is nothing to prove.

Thus $[N]=[M^*]$ and $M \otimes M^* \cong k \oplus (\proj)$ so the kernel of $\ev$ must be projective.
\end{proof}

It we relax the condition that $M$ be finite dimensional we do not get any more endotrivial modules. 

\begin{lemma}
\label{lem:infendo}
If $M$ and $N$ are possibly infinite dimensional $kG$-modules such that $M \otimes N \cong k \oplus (\proj)$ then $M = \overline{M} \oplus (\proj)$ for some finite dimensional indecomposable $kG$-submodule $\overline{M}$ that is endotrivial in the original sense.
\end{lemma}

\begin{proof}
Using the isomorphism $M \otimes N \cong k \oplus (\proj)$, write a generator of $K$ as $\sum m_i \otimes n_i$ and let $M'=\langle m_i \rangle_{kG} \subseteq M$. Then $\dim_k M' < \infty$ and $k \subseteq M' \otimes N \subseteq M \otimes N \rightarrow k$, so $k \mid M' \otimes N$. Tensoring with $M$, we obtain $M \mid M' \otimes N \otimes M \cong M' \oplus (\proj)$.

Somehow we deduce that $M= (\fd) \oplus (\proj)$. One way is to use an advanced version of the Krull-Schmidt Theorem (the right hand side is a sum of countably generated modules with local endomorphism rings), or see Exercise 7.

The same considerations apply to $N$, so we have $k \oplus (\proj) \cong M \otimes N = \overline{M} \otimes \overline{N} \oplus (\proj)$.
\end{proof}

The proof of Lemma~\ref{lem:invdual} now works for infinite dimensional modules too.

Here are some examples of calculations for finite $p$-groups in characterisitic $p$.

\noindent
$T(C_{p^n}) = \begin{cases} 0 \quad \mbox{if $p^n=2$} \\ \mathbb Z/2 \quad \mbox{if $p^n \ne 2$} \end{cases}$\\
$T(Q_8) = \begin{cases} \mathbb Z/4 \quad \mbox{if $k$ does not contain a cube root of unity} \\ \mathbb Z/4 \oplus \mathbb Z/2 \quad \mbox{if $k$ contains a cube root of unity} \end{cases}$\\
$T(Q_{2^n}) = \mathbb Z/4 \oplus \mathbb Z/2$ for $2^n \geq 16$, regardless of cube roots\\
$T(D_{2^n}) = \mathbb Z \oplus \mathbb Z$ for $2^n \geq 8$

Always one generator is $\Omega k$. For $Q_8$, when $k$ contains a cube root of unity $\omega$, there is a module described by $i= \left[ \begin{smallmatrix} 1 & 0&0 \\ 1 & 1&0 \\ 0&1&1 \end{smallmatrix} \right], j= \left[ \begin{smallmatrix} 1 & 0&0 \\ \omega & 1 &0 \\ 0& \omega^2 &1 \end{smallmatrix} \right]$. For $D_{2^n}$, another generator is the kernel of the natural map $k[D_{2^n}/C_2] \rightarrow k$, where $C_2$ is not central.

It is known that $T(G)$ is always a finitely generated group, but it can be very difficult to calculate. The classification of endotrivial modules for any finite $p$-group was a major achievement of Carlson and Th\'{e}venaz.

\begin{theorem}
If all maximal elementary abelian subgroups of the $p$-group $P$ have rank at least 3 then $T(P) \cong \mathbb Z$, generated by $\Omega k$.
\end{theorem}

\begin{theorem} 
Suppose that the $p$-group $P$ has at least one maximal elementary abelian subgroup of rank 2 and $P$ is not semi-dihedral. Then $T(P)$ is free abelian on $r$ generators, where $r$ is defined by letting $c$ be the number of conjugacy classes of elementary abelian subgroups of rank 2 and setting $r=c$ if $\rank (P)=2$ and $r=c+1$ if $\rank (P) >2$.
\end{theorem}

The other cases are dealt with separately. Extraspecial and almost extraspecial groups are particularly tricky. For more background and details see the survey article \cite{thevenaz}.

\begin{exercises}
Here $G$ is a finite group and $kG$-modules are finite dimensional unless stated otherwise.
\begin{enumerate}
\item
Show that the natural map $M \otimes _k M^* \rightarrow \End_k(M)$, $n \otimes f \mapsto ( m \mapsto f(m)n)$ is equivariant for any $M$ and an isomorphism when $M$ is finite dimensional. Thus a module $M$ is endotrivial if and only if $\End_k(M) \cong k \oplus (\proj)$. This explains the terminology \emph{endotrivial}.

\item
If $M$ is endotrivial and $M \cong A \oplus B$ show that either $A$ is projective or $B$ is projective. Deduce that $M$ is of the form $(\indec) \oplus(\proj)$.

\item
Verify that $(\proj) \otimes_k (\anything) = (\proj)$, even for infinite groups and modules. (Hint: see any textbook.)

\item
Verify that $\Hom _k ( (\proj), (\anything)) = ( \inj )$, even for infinite groups and modules. (Hint: $I$ is injective $ \Leftrightarrow$ $\Hom_{kG}(-,I)$ is exact; use $\Hom_{kG}(A,\Hom_k(B,C)) \cong \Hom_k(A \otimes _{kG} B, C)$.)

\item
If $M$ is an endotrivial $kG$-module, show that $X \mapsto X \otimes M$ defines an autoequivalence of the category $kG \mhyphen \Mod$.

\item
Let $G = H \times F$, where $F$ is a $p'$-group and $p$ divides the order of $H$. Show that $T(G)=T(H) \times \Hom_{\Grp}(F,k^{\times})$. Hint: if $[M] \in T(G)$ and $[\Res^G_HM]=[k]$ in $T(H)$, consider the Tate cohomology group $\widehat{H}^0(H;M)$ as a $kG$-module.

\item
Here $M$ may be infinite dimensional. Define $M_{\np}= \mho \Omega M$, where $\Omega$ and $\mho$ are calculated using the projective cover and the injective hull respectively. Show that $M_{\np}$ has no projective summands. Show that there is a natural map $M_{\np} \rightarrow M$, it is injective and the cokernel is projective. Thus $M \cong M_{\np} \oplus (\proj)$. Show also that $(M \oplus N)_{\np} \cong M_{\np} \oplus N_{\np}$. Use this to show that if $M$ is a summand of a module of the form $(\fd) \oplus (\proj)$ then $M$ is also of this form. 
\end{enumerate} 
\end{exercises}

\section{Groups of Type $\Phi$}
\label{sec:phi}

Now we allow groups to be infinite. 

\begin{definition}
The projective dimension of a $kG$-module $M$, denoted by $\projdim_{kG}$, is the length of the shortest possible projective resolution over $kG$, or $\infty$ if there is no resolution of finite length. The injective dimension is defined similarly.
\end{definition}

\begin{definition}
A group $G$ is said to be of type $\Phi$ (over $k$) if, for any $kG$-module $M$, the restriction $M \! \downarrow _F$ to any finite subgroup $F$ is of finite projective dimension then $M$ has finite projective dimension (over $kG$).
\end{definition}

Note that, since here we are taking $k$ to be a field of characteristic $p$, finite projective dimension over $F$ is equivalent to projective and in any case we only need to check restrictions to finite elementary abelian $p$-subgroups (Chouinard's Theorem). 

\begin{definition}
The finitistic dimension of $kG$ is $\findim kG = \sup \{ \projdim_{kG} M \mid \projdim_{kG} < \infty \}$.
\end{definition}

For groups of type $\Phi$, $\findim kG < \infty$. For otherwise, for each $i \in \mathbb N$, let $M_i$ be a $kG$-module with $i \leq \projdim_{kG}M < \infty$. Consider $M=\oplus _{i \in \mathbb N}M_i$; $M \! \downarrow _F$ is projective for any finite subgroup but $\projdim_{kG}M = \infty$.

\begin{proposition}
\label{prop:phi}
Suppose that the group $G$ acts on a contractible CW-complex of finite dimension by permuting the cells and with finite cell stabilizers. Then $G$ is of type $\Phi$.
\end{proposition}

\begin{proof}
Let $X$ be the CW-complex, with cellular chain complex $C_*(X)$. 

\[ \begin{tikzcd} C_n(X) \ar[r] & \cdots \ar[r] & C_1(X) \ar[r] & C_0(X) \ar[d] \\
&&& k
\end{tikzcd} \]
Each $C_i(X)$ is a sum of permutation modules $k \! \uparrow _F^G$, $F$ finite (the notation means induced module). Tensor with $M$.
\[ \begin{tikzcd} C_n(X) \otimes M \ar[r] & \cdots \ar[r] & C_1(X) \otimes M \ar[r] & C_0(X) \otimes M \ar[d] \\
&&& M
\end{tikzcd} \]
Since $k \! \uparrow _F^G \otimes M \cong M \! \downarrow_F \! \uparrow _F^G$, if each $M \! \downarrow _F$ is projective then this is a projective resolution of $M$.
\end{proof}

\begin{definition}
A group $G$ is of finite virtual cohomological dimension over $k$ (finite $\vcd_k$) if it has a subgroup $H$ of finite index such that $\projdim_{kH}k < \infty$. 
\end{definition}

Note that this subgroup $H$ can contain no $p$-torsion.

\begin{theorem}
If $G$ is of finite $\vcd$ then it is of type $\Phi$. 
\end{theorem}

This theorem is essentially due to Serre, see \cite{brown}. When $k= \mathbb Z$ you actually construct a CW-complex as in Proposition~\ref{prop:phi}. Otherwise you use algebraic methods to construct an analogue of the associated chain complex, see \cite{brown, swan}.

It follows that groups such as $\Sl_n(\mathbb Z)$ or, more generally, discrete subgroups of Lie groups with finitely many components have finite $\vcd$ (over $\mathbb Z$, hence over any $k$), see \cite{brown}.

However $(\mathbb Z/p)^{(\mathbb N)}$ and $\mathbb Q / \mathbb Z$ are not of finite $\vcd_k$, because there is no subgroup of finite index with no $p$-torsion. They can both act on a tree with finite stabilizers, so they are of type $\Phi$. On the other hand, $\mathbb Z^{(\mathbb N)}$ is not even of type $\Phi$ (Exercise 4 of this section).

We now define a category $\ModProj(kG)$ in which all the projective modules are identified with 0. We say that a morphism $f \! : M \rightarrow N$ factors through a projective if there is a projective module $P$ such that $f$ factors as a composition $M \rightarrow P \rightarrow N$. We define $\ModProj(kG)$ to have the same objects as $kG \mhyphen \Mod$, but $$\Hom _{\ModProj(kG)}(M,N)=\Hom _{kG}(M,N)/(\mbox{factors through a projective module}).$$ In this catgory the syzygy $\Omega M$ is well defined up to unique isomorphism (Schanuel's Lemma) and there is a map $\Omega  \!: \Hom _{\ModProj(kG)}(M,N) \rightarrow \Hom _{\ModProj(kG)}(\Omega M,\Omega N)$ that is well defined by the diagram

\[ \begin{tikzcd}
\Omega M \ar[r] \ar[d, dashed, "{\Omega f}"] & P_M \ar[r] \ar[d, dashed] & M \ar[d, "f"] \\
\Omega N \ar[r] & P_N \ar[r] & N.
\end{tikzcd} \]
Thus $\Omega$ becomes a functor from $\ModProj(kG)$ to itself.

Two $kG$-modules $M$ and $N$ are isomorphic in $\ModProj(kG)$ if and only if there exist projective modules $P$ and $Q$ such that $M \oplus P \cong N \oplus Q$ in $kG \mhyphen \Mod$ (exercise).

For finite groups, $\ModProj(kG)$ is considered to be the stable category, but for infinite groups it is not satisfactory, because not every module can be embedded in a projective module so we are unable to invert $\Omega$.

\begin{definition} Define the stable category of $kG$-modules to have the same objects as $kG \mhyphen \Mod$ and morphisms
$\Hom_{\Stab(kG)}(M,N)= \dirlim_{\Omega} \Hom _{\ModProj(kG)}(\Omega^r M,\Omega^r N)$. 
\end{definition}

For the rest of this section, all groups are of type $\Phi$ and there are no restrictions on the modules.

\begin{lemma}
\label{lem:spli}
Any injective $kG$-module has projective dimension at most $\findim (kG)$.
\end{lemma}

\begin{proof}
The restriction of an injective module to any finite subgroup is injective (the left adjoint of restriction is induction, which is exact), which is equivalent to projective for a finite group. Now use the definition of type $\Phi$.
\end{proof}

The next proposition is from \cite{GG} and its proof is a little tricky, but it will be crucial later. If you are interested in what happens when you consider a more general ring $k$ than a field, this is where the noetherian condition is useful.

\begin{proposition}
\label{prop:silp}
Any projective $kG$-module has injective dimension at most $\findim (kG)$.
\end{proposition}

\begin{proof}
Let $P$ be a projective module. There is a map $k \rightarrow kG^*$ (the dual of $kG$) that sends 1 to the augmentation map $\epsilon : kG \rightarrow k$. Applying $\Hom_k(-,P)$ yields a surjection $\Hom_k(kG^*,P) \rightarrow P$, which splits since $P$ is projective. This allows us to replace $P$ by $\Hom_k(kG^*,P)$ (using the injective version of Exercise 5 of this section).

The module $kG^*$ is injective by Exercise 4 of Section 1, so has a projective resolution $Q_{*}$ of length at most $\findim (kG)$ by the previous lemma. Applying $\Hom_k(-,P)$ (which is exact), we obtain a quasi-isomorphism from $\Hom_k(kG^*,P)$ to $\Hom_k(Q_{*},P)$. Each $\Hom_k(Q_i,P)$ is injective, by the same exercise.
\end{proof}

\begin{exercises}
\begin{enumerate}
\item
Show that if $G$ is of type $\Phi$, then so is any subgroup. (Hint: consider induced modules.)
\item
In Proposition~\ref{prop:phi} we didn't really use the CW-complex, only its chain complex. Write down a version that only uses the existence of 
 an exact sequence of certain types of $kG$-modules.
\item
Show that if $G$ acts cellularly on a CW-complex of finite dimension $d$ and there is a number $e$ such that every stabilizer $H$ is of type type $\Phi$ and has $\findim (kH) \leq e$, then $G$ is of type $\Phi$ and $\findim (kG) \leq d+e$.
\item
Show that $\mathbb Z ^{(\mathbb N)}$ has infinite finitistic dimension, so is not of type $\Phi$. (Hint: projective dimension cannot increase when you restrict to a subgroup.)
\item
If $A \rightarrow B \rightarrow C$ is a short exact sequence of modules show that:\\
$\projdim A \leq \max \{ \projdim B, \projdim C-1 \}$,\\
$\projdim B \leq \max \{ \projdim A, \projdim C \}$,\\
$\projdim C \leq \max \{ \projdim A+1, \projdim B \}$,\\
$\projdim D \oplus E = \max \{ \projdim D, \projdim E \}$.\\
\end{enumerate}
\end{exercises}
 
\section{Complete Resolutions}

\begin{definition}
A complete resolution of a $kG$-module $M$ is a commutative diagram
\[ \begin{tikzcd}
\cdots Q_{n+1} \ar[r] \ar[d, equal] & Q_n \ar[r, "d_n"] \ar[d, equal] & Q_{n-1} \ar[r, "d_{n-1}"] \ar[d] & \cdots  \ar[r] & Q_1 \ar[r, "d_1"] \ar[d] & Q_0 \ar[r, "d_0"] \ar[d] & Q_{-1} \ar[r] & Q_{-2}  \cdots \\
\cdots P_{n+1} \ar[r] & P_n \ar[r] & P_{n-1} \ar[r] & \cdots  \ar[r] & P_1 \ar[r] & P_0  \ar[d] \\
&&&&&M,
\end{tikzcd} \] 
where the $P_i$ and $Q_i$ are projective, $P_{*}$ is a projective resolution of $M$ and $Q_{*}$ is acyclic (i.e.\ exact). The integer $n$ is called the coincidence index.

We also require the extra condition that $\Hom_{kG}(Q_{*},P)$ be acyclic for any projective module $P$.
\end{definition}

For $n=0$ this is the same as the definition used for the Tate cohomology of finite groups.

\begin{theorem}
For a group of type $\Phi$, any $kG$-module has a complete resolution with coincidence index at most $\findim (kG)$. Any two complete resolutions of the same module are chain homotopy equivalent. A homomorphism of modules induces a chain map between complete resolutions, unique up to chain homotopy.
\end{theorem}

We will now sketch the construction of these complete resolutions. From now on all groups will be understood to be of type $\Phi$.

For the next result we will write $\Omega^n M$ to denote a module that is the $n-1$st kernel in some projective resolution of $M$, even though this is only defined up to projective summands.

\begin{lemma}
If $n \geq \findim (kG)$ then a module of the form $\Omega ^{n}M$ can be embedded in a projective module in such a way that the quotient is also of the form $\Omega^{n}N$, for some $kG$-module $N$.
\end{lemma}

\begin{proof}
Embed $M$ in an injective module $I$, with quotient $N$, say. By hypothesis, we already have the projective resolution $P_{*}$ of $M$ used to construct $\Omega^{n}M$ and we find a projective resolution $R_{*}$ of $N$. By the Horseshoe Lemma, these can be combined to give a projective resolution of $I$ such that all rows and columns in the following diagram are exact.
\[ \begin{tikzcd}
\Omega^{n}M \ar[r] \ar[d] & P_{n-1} \ar[r] \ar[d] & \cdots \ar[r] & P_0 \ar[r] \ar[d] & M \ar[d] \\ 
\Omega^{n}I \ar[r] \ar[d] & P_{n-1} \oplus R_{n-1} \ar[r] \ar[d] & \cdots \ar[r] & P_0 \oplus R_0 \ar[r] \ar[d] & I \ar[d] \\
\Omega^{n}N \ar[r]  & R_{n-1} \ar[r]  & \cdots \ar[r] & R_0 \ar[r] &  N  
\end{tikzcd} \]
By Lemma~\ref{lem:spli}, $\Omega^{n}I$ is projective.
\end{proof}

We can construct $Q_{*}$ as follows. Set $n = \findim (kG)$. For $i \geq n$ we take $Q_i=P_i$. Now apply the above lemma to $\Omega^{n}M$ to embed it in a projective module that we take as $Q_{n-1}$. The quotient is also of the form $\Omega^{n}N$ and we can keep repeating the process.

\begin{lemma}
\label{lem:homX}
The complex $\Hom_{kG}(Q_{*},X)$ is exact if $X$ has finite injective dimension or if $X$ is projective.
\end{lemma}

\begin{proof}
In view of Proposition~\ref{prop:silp}, the second part follows from the first. We prove the first by induction on the injective dimension of $X$. When this is 0 the result is clear. Otherwise we embed $X$ in an injective module to obtain a short exact sequence $X \rightarrow I \rightarrow Y$ with $Y$ of lower injective dimension. Consider the double complex of $\Hom_{kG}$ from $Q_{*}$ to the short exact sequence just produced, written so that $Q_{*}$ is horizontal. All the columns are exact, because the $Q_i$ are projective. Two of the rows are exact by induction, thus so is the remaining one, by a diagram chase.
\end{proof}  

This lemma shows that the extra condition in the definition of a complete resolution holds automatically in the case of groups of type $\Phi$. Thus we could have omitted it from the definition, but in more general contexts it should be part of the definition so we have included it here. This extra condition is very important, because we can use it to produce morphisms. It yields the remaining vertical arrows in the definition of a complete resolution. It is also what is needed to produce the maps used to show that any two complete resolutions are chain homotopy equivalent. For more details see \cite{ikenaga}.

If two acyclic complexes of projectives are chain homotopy equivalent then their kernels are isomorphic up to projective summands (Exercise 2).

\begin{definition}
From now on, given a $kG$-module $M$ and $i \in \mathbb Z$ we will use $\Omega^iM$ to denote the kernel of $d_{i-1}$ in a complete resolution of $M$. This is well defined as a module only up to projective summands and an isomorphism unique in $\ModProj(kG)$. For small non-negative $i$ it only agrees with the previous definition stably. In particular, we have negative syzygies.
\end{definition}

\begin{definition}
The modules that can occur as a kernel in an acyclic complex of projectives $Q_{*}$ such that $\Hom_{kG}(Q_{*},P)$ is acyclic for any projective module $P$ are called Gorenstein projective.
\end{definition}

Gorenstein projective modules have many nice properties (see the exercises). In many cases they are easy to recognise: for example, for a group of finite $\vcd_k$, a module is Gorenstein projective if and only if it projective on restriction to some subgroup of finite index (Exercise 3). 

Notice that $\Omega^0M$ is not the same as $M$. From the defining diagram for a complete resolution we see that there is a natural map $\Omega^0M \rightarrow M$; it is a stable isomorphism because we can see in the diagram that a large enough syzygy of it is just equality. We sometimes write this as $\epsilon : \tilde{M} \rightarrow M$. Note that $\tilde{M}$ is Gorenstein projective and $\epsilon$ is a stable equivalence. In particular, every module is stably equivalent to a Gorenstein projective module.

\begin{theorem}
\label{thm:cat}
For a group $G$ of type $\Phi$, the following functors are equivalences of categories and going around the loop is isomorphic to the identity.
\end{theorem}
\vspace{-1.5em}
\[ \begin{tikzcd}
\left(\mbox{$kG$-modules}, \ \Hom_{\Stab(kG)}\right) \ar[d, "\deg 0"] && \left(\mbox{Gorenstein projective $kG$-modules},\ \Hom_{\ModProj(kG)}\right) \ar[ll, "\mbox{\footnotesize{inclusion}}"'] \\
D^b(kG \mhyphen \Mod)/D^b(kG \mhyphen \Proj) \ar[rr, "\mbox{\footnotesize{complete}}", "\mbox{\footnotesize{resolution}}"'] && \left(\mbox{acyclic complexes of projectives},\ \mbox{chain homotopy}\right) \ar[u, "\Omega^0"]
\end{tikzcd} \]

Here $\deg 0$ means that we consider a module as a complex consisting of just that module in degree 0 and 0 elsewhere. The complete resolution of a bounded complex is constructed in a way similar to that of a module. The complex has an ordinary projective resolution. We start constructing the complete resolution starting with $\Ker(d_{n-1})$ in this projective resolution, where $n$ is bigger than $\findim (kG)$ plus the degree of the top non-zero term of the complex.

$D^b(kG \mhyphen \Mod)$ is the derived category of bounded complexes of $kG$-modules. It is easy to see that we get the same category if we allow complexes that are only bounded on the right (in the direction of the arrows) but have only finitely many non-zero homology groups. $D^b(kG \mhyphen \Proj)$ is the derived category of bounded complexes of projective modules; these modules may be infinitely generated,  so this is not what is called the category of perfect complexes. This is equivalent to $K^b(kG \mhyphen \Proj)$, where the morphisms are taken up to chain homotopy.

We consider any of these categories to be the stable module category of $kG$.

Apart from the Gorenstein projective modules, the other three categories in this theorem have a natural triangulated structure. The shift functor is shift in degree for the categories on the bottom row and $\Omega^{-1}$ on the top row. These equivalences preserve the triangulated structure and the tensor product over $k$. See \cite{happel} for more background on derived and triangulated categories.

\begin{remark}
The equivalence of these categories is originally due to Buchweitz \cite{buchweitz}, although the context is slightly different. It should be clear from the proof that these equivalences still hold when $kG$ is replaced by any ring such that the projective length of any injective module is finite and the injective length of any projective module is finite. In fact, these two conditions are also necessary \cite{beligiannis} and for a group ring $kG$, for \emph{any} group $G$, one implies the other \cite{emmanouil}.

Complete resolutions appeared earlier, in the context of Tate-Farrell cohomology; see \cite{brown, ikenaga}.
\end{remark}

We conclude this section with a couple of lemmas that will be useful later.

\begin{lemma}
\label{lem:GPdom}
If $M$ and $N$ are $kG$-modules with $M$ Gorenstein projective then any stable map $f \! :M \rightarrow N$ can be realized as a genuine homomorphism of modules.
\end{lemma}

\begin{proof}
See Exercise 6 of this section.
\end{proof}
 
Now that we are working in a triangulated category we can formulate a very useful property of groups of type $\Phi$. 

Recall that if  $ X \stackrel{f}{\rightarrow} Y \rightarrow  Z \rightarrow $  is a triangle then $f$ is an isomorphism if and only if $Z \simeq 0$.

\begin{lemma}
\label{lem:stabiso}
A stable morphism $f \! :X \rightarrow Y$ is a stable isomorphism if and only if $f \! \downarrow _P:X \! \downarrow _P \rightarrow Y \! \downarrow _P$ is a stable isomorphism for every finite (elementary abelian) $p$-subgroup $P$.
\end{lemma}

\begin{proof}
Complete the triangle with a third module $Z$. The definition of type $\Phi$ shows that $Z=0$.
\end{proof}

\begin{exercises}

\begin{enumerate}
\item
Show that if $Q_{*}$ is a complete resolution of $k$ then, for any module $M$, $Q_{*} \otimes k$ is a complete resolution of $M$.
\item
\begin{enumerate}
\item
Verify that two complete resolutions of the same module must be chain homotopy equivalent.
\item
Verify that two acyclic complexes of projective modules that are chain homotopy equivalent have isomorphic kernels up to projective summands, by an isomorphism the same in $\ModProj(kG)$ as the map induced by the chain map.
\end{enumerate}
\item
Suppose that $G$ is of finite $\vcd_k$. Show that the following conditions on a module $M$ are equivalent:
\begin{enumerate}
\item
$M$ is Gorenstein projective;
\item
for some subgroup $H$ of finite index, $M\! \downarrow_H$ is projective;
\item
for any subgroup $H$ of finite index and $p$-torsion free, $M\! \downarrow_H$ is projective.
\end{enumerate}
(Hint: Induction is the right adjoint of restriction for subgroups of \emph{finite index}, so there is an embedding $M \rightarrow M \! \downarrow_H \! \uparrow^G$; see \cite{brown} for more details, particularly for the last part.)
\item
If $G$ has a normal subgroup $N$ that is $p$-torsion free and both $G$ and $N$ are of type $\Phi$, show that inflation gives a well-defined functor on stable categories $\Stab(kG/N) \rightarrow \Stab(kG)$. Convince yourself that this will not work if $N$ contains $p$-torsion, even if $G$ is finite.
\item
\begin{enumerate}
\item
Show that (Gorenstein projective)$\otimes$(anything) = (Gorenstein projective).
\item
Show that Gorenstein projective and finite projective dimension implies projective.
\end{enumerate}
\item
For a group $G$ of type $\Phi$ and with $\findim (kG) \leq d$ and a $kG$-module $M$, show that the following conditions are equivalent:
\begin{enumerate}
\item
$M$ is Gorenstein projective;
\item
there is an exact sequence $0 \rightarrow M \rightarrow P_{m-1} \rightarrow \cdots \rightarrow P_0 \rightarrow X \rightarrow 0$, for some $m \geq d$, some $X$ and $P_i$ projective;
\item
there is a projective module $P$ such that $M \oplus P$ is Gorenstein projective;
\item
$M$ is a summand of a Gorenstein projective module;
\item
for all $m,n \geq 0$ and all modules $N$, the map\\ $\Omega^m: \Hom_{\ModProj}(\Omega^nM, \Omega^nN) \rightarrow  \Hom_{\ModProj}(\Omega^{m+n}M, \Omega^{m+n}N)$\\ is an isomorphism;
\item
for all modules $N$, the map $\Omega^\infty: \Hom_{\ModProj}(M,N) \rightarrow  \Hom_{\Stab(kG)}(M,N)$ is an isomorphism; 
\item
for all projective modules $P$ and all $i \geq 1$, $\Ext^i_{kG}(M,P)=0$;
\item
$M \otimes N$ is projective whenever $\projdim N < \infty$.
\end{enumerate}
\item
Let $Q_{*}$ be a complete resolution of $M$ and define the complete $\Ext$-groups by\\ $\widehat{\Ext}^i_{kG}(M,N)= H^i( \Hom_{kG}(Q_{*},N))$ and complete cohomology by $\widehat{H}^i(G;k) = \widehat{\Ext}^i_{kG}(k,k)$. Check the following (see \cite{brown}):
\begin{enumerate}
\item
$\widehat{\Ext}^i_{kG}(M,N)$ does not depend on the complete resolution chosen,
\item
$\widehat{\Ext}^i_{kG}(M,N) \cong \Ext^i_{kG}(M,N)$ for $i > \findim (kG)$,
\item
$\widehat{\Ext}^0_{kG}(M,N) \cong \Hom_{\ModProj}(\tilde{M},N) \cong \Hom_{\Stab (kG)}(M,N)$.
\end{enumerate}
\end{enumerate}
\end{exercises}

\section{Endotrivial Modules for Infinite Groups}

We repeat our standing assumption that all the groups we consider are of type $\Phi$.

\begin{definition}
A $kG$-module $M$ is endotrivial if there is another module $N$ such that $M \otimes N \simeq k$ in $\Stab(kG)$.
\end{definition}

The stable isomorphism classes of endotrivial modules form a group $T(G)$.

\begin{proposition}
\label{prop:subend}
If $M$ is endotrivial then its inverse is its dual $M^*$.

The module $M$ is endotrivial if and only if $M \! \downarrow _P$ is endotrivial for every finite (elementary abelian) $p$-subgroup $P$.
\end{proposition}

\begin{proof}
Clearly an endotrivial module remains endotrivial on restriction. Consider the evaluation map $M \otimes M^* \stackrel{\ev}{\rightarrow} k$. If $M$ is endotrivial on restriction to a finite $p$-subgroup $P$, then $\ev$ is a stable isomorphism over $kP$, by the infinite dimensional version of Lemma~\ref{lem:invdual}, see the remark after Lemma~\ref{lem:infendo}. Thus $\ev$ is a stable isomorphism, by Lemma~\ref{lem:stabiso}.
\end{proof}

\begin{example}
$G = C_p \ast C_p'$, the free product of two groups of order $p$. $G$ acts on a tree with stabilizers conjugate to either $C_p$ or $C_p'$, so it is of type $\Phi$.

Consider the canonical map $k \! \uparrow ^G_{C_p} \rightarrow k, \ g \otimes x \mapsto gx$. If we restrict this to $C_p$ and use the Mackey formula we get a split surjection $k \oplus (\free) \rightarrow k$ and if we restrict it to $C_p'$ we get $(\free) \rightarrow k$. We also do the same thing starting with $C_p'$ and combine the two to get a map $k \! \uparrow ^G_{C_p} \oplus k \! \uparrow ^G_{C_p'} \rightarrow k$. On restriction to either $C_p$ or $C_p'$ this is a stable isomorphism. Any torsion subgroup of $G$ is conjugate to one of these two. Applying Lemma~\ref{lem:stabiso}, we obtain a stable isomorphism \[ k \simeq k \! \uparrow ^G_{C_p} \oplus k \! \uparrow ^G_{C_p'} . \]

Thus the trivial module decomposes. Note that the right hand side is Gorenstein projective by Exercise 3 of Section 3, since it is free over a subgroup of finite index.
\end{example}

In order to understand this decomposition better, we consider certain subgroup complexes associated to $G$. A suitable source for the theory of subgroup complexes is \cite{benson}, although it only deals with finite groups. Most of the proofs carry over easily to the infinite case, with a little help from \cite[II 2.7]{tomdieck}. This material is not used in what follows, so we will be brief. 

The Brown or Quillen complex is a simplicial complex, $\Delta (G)$, where the $r$-simplices correspond to chains $P_0<P_1 < \cdots P_r$ of non-trivial $p$-subgroups in the case of the Quillen complex, denoted $|\mathcal S_p(G)|$, or non-trivial elementary abelian $p$-subgroups in the case of the Brown complex, $|\mathcal A_p(G)|$. $G$ acts on $\Delta (G)$ by conjugation. The Brown and Quillen complexes are known to be equivariantly homotopy equivalent, so it will not matter which one we use. For simplicity here, we are going to assume that $\Delta (G)$ is equivariantly homotopy equivalent to a finite dimensional complex. This is clearly the case  for $|\mathcal A_p(G)|$ if the $p$-rank of $G$ is finite (i.e.\ there is a bound on rank of any elementary abelian $p$-subgroup).

We are interested in the simplicial chain complex $C(\Delta (G))$ and its augmented version $C(\Delta (G)) \stackrel{\epsilon}{\rightarrow} k$, which we denote by $\tilde{C}(\Delta (G))$. It can be shown that for any non-trivial finite $p$-subgroup $P$ of $G$, the fixed point set $\Delta (G)^P$ is contractible. Quillen showed that it follows that $\tilde{C}(\Delta (G))$ restricted to such a $P$ is equivalent in the derived category to a bounded complex of projectives and Webb showed that in fact the restriction is homotopy equivalent to such a complex. 

We can regard this augmented complex as an element of the derived category and hence of the stable category $\Stab(kG)$, by Theorem~\ref{thm:cat}. For this we use the functor $\Omega ^0 \circ (\mbox{projective resolution})$, which we abbreviate to $\Omega^0$. The fact that $\tilde{C}(\Delta (G))$ is equivalent to a bounded complex of projectives after restriction to any finite $p$-subgroup $P$ means that it is 0 in $\Stab(kP)$. Thus it is 0 even in $\Stab(kG)$, by Lemma~\ref{lem:stabiso}. Since we are working in a triangulated category, it follows that the augmentation $\epsilon$ is an isomorphism between $C(\Delta (G))$ and $k$. At the level of stable  modules, this means that $\Omega^0C(\Delta (G)) \simeq k$.

\begin{theorem}
If $C(\Delta (G))$ is equivariantly homotopy equivalent to a finite dimensional complex, where $\Delta (G)$ one of the complexes defined above (e.g.\ if $p \mhyphen \rank (G) < \infty$), then
$\Omega^0C(\Delta (G)) \simeq k$. Thus $k$ decomposes stably as a sum, $k \simeq \oplus_e k_e$, one for each path component of $\Delta (G)/G$. 
\end{theorem}

It is easy to see that the path components of $|\mathcal S_p(G)|/G$ correspond to the equivalence classes of non-trivial $p$-subgroups of $G$ under the equivalence relation generated by setting $P \sim Q$ if $P$ is contained in $Q$ or $P$ is conjugate to $Q$.
A more homological version of this result appears in \cite{CL}.

Let us write $\widehat{\End}_G(k)$ or $\widehat{\Aut}_G(k)$ for the stable endomorphism or automorphism group of $k$. Then $\widehat{\End}_G(k) = \prod_e \widehat{\End}_G(k_e)$ and we obtain idempotents corresponding to the $k_e$, which we also denote by $e$.
In fact, these idempotents are primitive (Exercise 2 of this section). 
It follows that we have an ``endo-$e$'' group $T_e(G)$ for each idempotent $e$, with identity $k_e$, and $T(G)= \prod_e T_e(G)$. 

There is a well-known fact that is important here.

\begin{theorem}
$\widehat{\End}_G(k)$ is a commutative ring.
\end{theorem}

\begin{proof} Given $f,g \in \widehat{\End}_{G}(k)$ there are two possible ways to form a product. The obvious way is to compose them, $f \circ g$; but since $k \otimes k \simeq k$ we can also take the tensor product, $f \otimes g$, which is commutative. We want to show that these two products agree.

A homological approach is to note that $\widehat{\End}_G(k) \cong \widehat{H}^0(G;k)$. It is shown in \cite{ikenaga} that $\widehat{H}^*(G;k)$ has a cup product, which agrees with the tensor product on $\widehat{H}^0(G;k)$. The cup product is known to agree with the composition product, see \cite{brown}.

Alternatively, show that, given four endomorphisms,  $d,e,f,g$, there is a relation $(d \circ e) \otimes (f \circ g) = (d \otimes f) \circ (e \otimes g)$ (draw a diagram with arrows). The result now follows formally: look up the Eckmann-Hilton argument.
\end{proof}

\begin{exercises}

\begin{enumerate}
\item
Show that if $M$ is endotrivial then then the natural map $M \otimes M^* \rightarrow \End_k(M)$, $m \otimes f \mapsto ( n \mapsto f(n)m)$ is a stable isomorphism, so $\End_k(M) \simeq k$ in the stable category.
\item
Let $e \in \widehat{\End}_G(k)$ be an idempotent. Show that if $P \leq G$ is a non-trivial finite $p$-subgroup then $\res ^G_Pe$ is either 0 or 1. Show also that if $P'$ is in the same component of $|\mathcal S_p(G)|/G$ as $P$ then $\res ^G_Pe=\res ^G_{P'}e$. Deduce that the idempotents corresponding to the components of $\Delta (G)/G$ are primitive.
\item
Calculate $\widehat{\End}_{C_p \times \mathbb Z}(k)$ and $\widehat{\Aut}_{C_p \times \mathbb Z}(k)$. One way is to construct a complete resolution by tensoring a complete resolution for $C_p$ with a projective resolution for $\mathbb Z$. Alternatively, use the spectral sequence $H^p(G/H;\widehat{H}^q(H;k)) \Rightarrow \widehat{H}^{p+q}(G;k)$ for $G$ of finite $\vcd_k$ and $H$ a normal subgroup such that $G/H$ is of finite $\cd_k$ \cite{brown}.
\end{enumerate}
\end{exercises}

\section{Groups Acting on Trees}

A group that acts on a tree is of type $\Phi$ if all the stabilizers are of type $\Phi$ and there is a bound on their finitistic dimensions, by Exercise 3 of Section 2. We will assume that this is the case. For more information about groups acting on trees, see \cite{serre}.

We will consider two basic examples, amalgamated free products and HNN extensions and will leave it to the reader to formulate the general result. We will obtain an exact sequence that allows us to calculate $T(G)$ for these groups. It is similar to part of the one used to calculate cohomology \cite{brown}.

\begin{theorem}
\label{thm:amal}
For an amalgamated free product $G=A * _C B$ there is an exact sequence
\[ \begin{tikzcd}[column sep=large] \widehat{\Aut}_G(k) \ar[r, "\res^G_A \times \res^G_B"] & \widehat{\Aut}_A(k) \times \widehat{\Aut}_B(k) \ar[r, "\res^A_C-\res^B_C"] \ar[draw=none]{d}[name=X, anchor=center]{} & \widehat{\Aut}_C(k) \ar[rounded corners,
            to path={ -- ([xshift=2ex]\tikztostart.east)
                      |- (X.center) \tikztonodes
                      -| ([xshift=-2ex]\tikztotarget.west)
                      -- (\tikztotarget)} ]{dll}[at end]{\delta}\\
                        T(G) \ar[r, "\res^G_A \times \res^G_B"] & T(A) \times T(B) \ar[r, "\res^A_C-\res^B_C"] & T(C). \end{tikzcd} \]
\end{theorem}

\begin{proof}
We will outline two approaches. One involves a natural representation-theoretic construction, but it is tricky to justify every step in the proofs of its properties. The other is more category-theoretic, but the proofs are formalities. 

For the first, assume that we are given a $kA$-module $M$ and a $kB$-module $N$ and an isomorphism of the restrictions $M \! \downarrow_C \simeq N \! \downarrow _C$ (stable isomorphism). Then we can find representatives of the stable isomorphism classes such that $M \! \downarrow_C \cong N \! \downarrow _D$ (genuine isomorphism) (Exercise 1). Let $\phi : M \! \downarrow_C \rightarrow N \! \downarrow _D$ be such an isomorphism.

Define a $kG$-module $C(M,N;\phi)$ to be $M$ as a vector space and with $G$ action $*$ given by
$$a*m=am, \quad b*m = \phi^{-1}(b \phi (m)), \quad \quad \mbox{for } a \in A, b \in B, m \in M.$$
These agree on $C$ and so do define a $kG$-module. It is convenient to denote $M$ with this twisted action of $B$ by $\phi^*N$; there is a $kB$-isomorphism \begin{tikzcd} \phi^*N \ar[r, "\tilde{\phi}", "\cong"'] & N \end{tikzcd}.

Note that $C(M,N;\phi) \! \downarrow _A =M$ and $C(M,N;\phi)\! \downarrow _B = \phi^*N \cong N$. Thus, if $M$ and $N$ are endotrivial then so is $C(M,N;\phi)$; this is because, by Proposition~\ref{prop:subend}, we only have to check the restrictions to finite $p$-subgroups and any finite subgroup of $G$ is known to be conjugate to a subgroup of $A$ or $B$. This proves exactness at $T(A) \times T(B)$.

The map $\delta : \widehat{\Aut}_C(k) \rightarrow T(G)$ is defined by $\delta (\phi ) = C(k,k; \phi)$. It is not obvious that this is well defined, i.e.\ that it depends only on the stable class of $\phi$, so let us assume this for the moment. That $\delta$ is a homomorphism follows from the description of the product in $\widehat{\Aut}_G(k)$ in terms of tensor product. 

If $M$ is a $kG$-module and there are stable isomorphisms \begin{tikzcd} M \! \downarrow _A \ar[r, "\theta_A", "\simeq"'] & k \end{tikzcd} and \begin{tikzcd} M \! \downarrow _B \ar[r, "\theta_B", "\simeq"'] & k, \end{tikzcd} then restricting to $C$ and combining with $M \! \downarrow _A \! \downarrow _C = M \! \downarrow_B \! \downarrow _C$ yields a map $\phi = \theta_B \! \downarrow_C \theta _A^{-1} \! \downarrow _C \in \widehat{\Aut}(k)$. The $kA$-isomorphism \begin{tikzcd}M \ar[r, "\theta_A", "\simeq"'] & C(k,k;\phi) \end{tikzcd} is also a $kB$-isomorphism, by construction; this proves exactness at $T(G)$. Similarly, if we have a stable isomorphism  \begin{tikzcd} C(k,k;\phi) \ar[r, "\theta", "\simeq"'] & k \end{tikzcd}, then $\phi = ( \tilde{ \phi } ( \theta ^{-1}\! \downarrow _B))\! \downarrow _C \circ (\theta \! \downarrow _A ) \! \downarrow _C $, which proves exactness at $\widehat{\Aut}_C(k)$. 

A more category-theoretic approach is to define a module $D(M,N;\phi)$, depending on the same data as before, as follows. The diagram
\[ \begin{tikzcd} M \! \downarrow _C \! \uparrow ^G \ar[r] \ar[d, "\phi \! \uparrow^G"] & M \! \uparrow^G \\
N \! \downarrow _C \! \uparrow ^G \ar[r] & N \! \uparrow^G \end{tikzcd} \]
leads to a map
\[ \begin{tikzcd}[row sep=small] M \! \downarrow _C \! \uparrow ^G \ar[r] & M \! \uparrow^G \oplus N \! \uparrow ^G \\
g \otimes m \ar[r, mapsto] & (g \otimes m, g \otimes \phi (m) ). \end{tikzcd} \]
Let $D(M,N ; \phi )$ be the cone; it only depends on stable data.

It is a standard fact that $G$ can act on a tree with two orbits of vertices, stabilizers conjugate to $A$ and $B$, and one orbit of edges, stabilizers conjugate to $C$. The corresponding augmented chain complex is a short exact sequence
\[ \begin{tikzcd} k \! \uparrow ^G_C \ar[r] & k \! \uparrow ^G_A \oplus k \! \uparrow ^G_B \ar[r] & k. \end{tikzcd} \]
Taking the tensor product with  $C(M,N; \phi )$, we obtain a short exact sequence
\[ \begin{tikzcd}[row sep=small] M \! \downarrow _C \! \uparrow ^G \ar[r]  & M \! \uparrow ^G \oplus \phi ^*N \! \uparrow ^G \ar[r] & C(M,N; \phi ), \\
g \otimes m \ar[r, mapsto] & (g \otimes m, g \otimes m)  & {} 
\end{tikzcd} \]
so $C(M,N, \phi ) \simeq D(M,N; \phi)$.

It is possible to work only with $D(M,N; \phi)$, using the geometry of the action on the tree.
\end{proof}

There is a similar sequence for an HNN extension $G = H * _{(f,A)}$. This notation means that there is a subgroup $ A \leq H$ and an injective homomorphism $f : A \hookrightarrow H$ and $G = \Grp \langle H,t \mid tat^{-1} = f(a) \rangle $.

\begin{theorem}
\label{thm:hnn}
For an HNN extension $G= H * _{(f,A)}$ there is an exact sequence
\[ \begin{tikzcd}[column sep=huge] \widehat{\Aut}_G(k) \ar[r, "\res^G_H"] & \widehat{\Aut}_H(k) \ar[r, "\res^H_A-f^*\res^H_{f(A)}"] \ar[draw=none]{d}[name=X, anchor=center]{}  & \widehat{\Aut}_A(k) \ar[rounded corners,
            to path={ -- ([xshift=2ex]\tikztostart.east)
                      |- (X.center) \tikztonodes
                      -| ([xshift=-2ex]\tikztotarget.west)
                      -- (\tikztotarget)} ]{dll}[at end]{\delta}\\
                        T(G) \ar[r, "\res^G_H"] & T(H) \ar[r, "\res^H_A-f^*\res^H_{f(A)}"]  & T(A). \end{tikzcd} \]
\end{theorem}

\begin{proof}
As before, we sketch two approaches. For the first, we need to know that, given a $kH$-module $M$ and a stable isomorphism \begin{tikzcd} M \! \downarrow _A \ar[r, "\theta", "\simeq"'] & f^*M \! \downarrow _{f(A)} \end{tikzcd} (i.e.\ $\theta(am)=f(a)\theta (m)$), we can arrange representatives of the stable classes such that $\theta$ is a genuine isomorphism of modules.

Define $E(M; \theta )$ to be $M$ as a vector space and with $G$ action $*$ given by
$$h*m=hm, \quad t*m = \theta (m), \quad \quad \mbox{for } h \in H, m \in M.$$
It is easy to check that $(tat^{-1})* m = \theta(a \theta ^{-1} (m)) = f(a) \theta \theta^{-1}(m)=f(a)m$.
Any finite subgroup of $G$ is conjugate to a subgroup of $H$ and $ E(M, \theta ) \! \downarrow _H \cong M $, so if $M$ is endotrivial, then so is $E(M;\theta)$, by Proposition~\ref{prop:subend}.
The map $\delta : \widehat{\Aut}_C(k) \rightarrow T(G)$ is defined by $\delta (\theta ) = E(k; \theta)$.

For the category-theoretic approach, define a module $F(M;\theta)$, as follows. The diagram
\[ \begin{tikzcd}[column sep=small] g \otimes m \ar[d, mapsto] & M \! \downarrow _A \! \uparrow ^G \ar[rr] \ar[d] && M \! \uparrow^G \\
gt^{-1} \otimes \theta (m) & M \! \downarrow _{f(A)} \! \uparrow ^G \ar[rr] && M \! \uparrow^G \end{tikzcd} \]
leads to a map
\[ \begin{tikzcd}[row sep=small] M \! \downarrow _A \! \uparrow ^G \ar[r] & M \! \uparrow^G \\
g \otimes m \ar[r, mapsto] & g \otimes m - gt^{-1} \otimes \theta (m) ). \end{tikzcd} \]
Let $F(M; \theta )$ be the cone; it only depends on stable data. The rest of the proof is similar to the previous case and is left to the reader.
\end{proof}

\begin{lemma}
If $C$ is finite in Theorem~\ref{thm:amal} then the map $\delta$ is zero. If $H$ is finite in Theorem~\ref{thm:hnn} then the map $\delta$ is injective.
\end{lemma}

\begin{proof}
If $C$ is finite then $\widehat{\Aut}_C(k) \cong k^{\times}$ and similarly for $H$. In other words, the only automorphisms are multiplication by a scalar. Clearly the map preceding $\delta$ in Theorem~\ref{thm:amal} is surjective, as is $\res^G_H$ in Theorem~\ref{thm:hnn}.
\end{proof}

\begin{exercises}
\begin{enumerate}
\item
In the context of Theorem~\ref{thm:amal}, show that it is possible to find representatives for $M$, $N$ and $\phi$ such that $\phi$ is a module isomorphism. First choose $M$ and $N$ Gorenstein projective so, by Exercise 2 of Section 3 or directly, there exist projective $kC$-modules $P$ and $Q$ such that $\phi$ can be realized as an isomorphism $M \! \downarrow _C \oplus P \rightarrow N \! \downarrow _C \oplus Q$. The snag is that $P$ might not be the restriction of a projective $kA$-module. Let $F$ be a free $kG$-module of sufficiently high rank that $F \! \downarrow _C \oplus P \cong F \! \downarrow _C \cong F \! \downarrow _C \oplus Q$, by the Eilenberg trick; add the appropriate restrictions of $F$ to $M$ and $N$.
\item
Calculate $T(\Sl_2(\mathbb Z))$ at different primes ($\Sl_2(\mathbb Z) \cong C_6 *_{C_2}C_4$).
\item
Calculate $T(C_{p^2}*_{C_p}C_{p^2})$.
\item
There is an obvious surjection $C_4*_{C_2}C_4 \rightarrow Q_8$. Calculate the inflation map $T(Q_8) \rightarrow T(C_4*_{C_2}C_4)$.
\item
Calculate $T(C_p * \mathbb Z)$ and $T(C_p \times \mathbb Z)$ (the latter group is an HNN extension). What happens to the 1-dimensional representations of $\mathbb Z$? Are the groups finitely generated?
\item
Calculate $T(C_p \times \mathbb Z \times \mathbb Z)$. Can you identify explicit modules that generate?
\item
Calculate $T(C_{p^\infty})$ ($C_{p^\infty}$ is the $p$-torsion subgroup of $\mathbb Q / \mathbb Z$; it acts on a tree with finite stabilizers \cite{ikenaga}).
\end{enumerate}
\end{exercises}
%%%%%%%%%%%%%%%%%%%%%%%%%%%%%%%%%%%%%%%%%%%%%%%%%%%%%%%%%%%%%%%%%%%%%%%%%%%%%%%%

\end{document}